\documentclass{amsart}

\usepackage{amssymb}
\usepackage{amsthm}

\usepackage{setspace}
\usepackage{epsfig}  		

\usepackage{subfig}
\usepackage{graphicx}

\newtheorem{theorem}{Theorem}[section]

\newtheorem{lemma}[theorem]{Lemma}
\newtheorem{corollary}[theorem]{Corollary}

\numberwithin{equation}{section}

\theoremstyle{definition}
\newtheorem{definition}[theorem]{Definition}

\theoremstyle{remark}

\newcommand{\bD}{\mathbb{D}}
\newcommand{\bC}{\mathbb{C}}

\newcommand{\bN}{\mathbb{N}}

\newcommand{\cP}{\mathcal{P}}

\newcommand{\ra}{\rightarrow}
\newcommand{\pa}{\partial}
\DeclareMathOperator{\diam}{diam}

\begin{document}
\onehalfspacing

\title{A Jordan Curve that cannot be Crossed by Rectifiable Arcs on a set of Zero Length}


\author{Jack Burkart}
\address{Department of Mathematics, Stony Brook University, Stony Brook, NY 11794}
\email{jack.burkart@stonybrook.edu}





\begin{abstract}
We construct a Jordan curve $\Gamma \subset \bC$ so that for any rectifiable arc $\sigma$ with endpoints in distinct complementary components of $\Gamma$, $H^1(\sigma \cap \Gamma) > 0$.
\end{abstract}


 \maketitle 



\section{Introduction and Motivation}
A Jordan curve is an injective image of the unit circle in the plane. The Jordan Curve Theorem asserts that a Jordan curve separates the plane into exactly two complementary components, a bounded interior component and an unbounded exterior component. This gives Jordan curves rather simple topology, but their geometry can be exotic. We say an arc \textit{crosses} a Jordan curve $\Gamma$ if its endpoints belong in different complementary components of $\Gamma$. Let $H^1$ denote the Hausdorff $1$-measure. In this paper, we prove the following theorem, answering a question posed by Sauter in \cite{MS:twisted-b}.

\begin{theorem}
\label{Main}
There exists a Jordan curve $\Gamma$ so that if $\sigma$ is a rectifiable (finite length) arc crossing $\Gamma$, $H^1(\sigma \cap \Gamma) > 0$.
\end{theorem}

We would like to make some more general remarks about Theorem \ref{Main}. It is a straightforward calculation from the construction that $\Gamma$ will have positive area, but a simple argument using Fubini's theorem and integrating over a family of line segments shows that any $\Gamma$ satisfying the conclusion of Theorem \ref{Main} must have positive area.

By the Riemann mapping theorem and Caratheodory's theorem, any Jordan curve can be crossed by an arc of $\sigma$-finite length that intersects the Jordan curve at exactly one point. Such an arc is analytic everywhere except possibly at the point where it crosses the boundary. In this sense, the hypothesis of rectifiablility is the weakest hypothesis on $\sigma$ so that the conclusion of Theorem \ref{Main} will hold. 

Theorem \ref{Main} has been addressed in other contexts, both in the class of arcs $\sigma$ considered and the size of the intersection of $\sigma \cap \Gamma$. Examples of Jordan curves $\Gamma$ that cannot be crossed by line segments at only one point seem well known, see for example \cite{mathoverflow}. In \cite{BishopSegment}, Bishop constructs a Jordan curve so that if $\sigma$ is a line segment crossing $\Gamma$, the Hausdorff dimension of $\sigma \cap \Gamma$ is $1$, but the case of positive length is not addressed. Motivated by questions about ordinary differential equations, in \cite{PughWu}, Pugh and Wu show that Jordan curves that cannot be crossed by rectifiable arcs at exactly one point exist and are actually generic (in a Baire sense) in the space of Jordan curves in the plane (equipped with an appropriate metric). 

The curve $\Gamma$ also has connections to conformal welding. If $\Gamma$ is a Jordan curve in the plane, let $\Omega$ denote its bounded complementary component. Then there are conformal mappings $f: \bD \ra \Omega$ and $g: \bC \setminus \overline{\bD} \ra \bC \setminus \overline{\Omega}$ which induce a circle homeomorphism $h: g \circ f^{-1}: \pa \bD \ra \pa \bD$ which we call a \textit{conformal welding}. $\Gamma$ is called \textit{log-singular} if there is a Borel set $E$ so that $E$ and $h(S^1 \setminus E)$ both have zero logarithmic capacity (for the definition of logarithmic capacity, see Chapter III of \cite{Garnett}). Results of Buerling in \cite{BeurlingEnsembles} show that if $\varphi: \bD \ra \Omega$ is a conformal mapping onto a Jordan domain $\Omega$, and $\zeta \in S^1$.
$$\int_0^1 |\varphi'(r\zeta)| dr < \infty$$
except perhaps on a set of zero logarithmic capacity (see also 23(a), p. 127, in \cite{Garnett}). It follows that outside of this exceptional set, the image of the hyperbolic geodesic rays $[0,\zeta)$ have finite length. Applying this result to both $f$ and $g$, it follows from Theorem \ref{Main} that $\Gamma$ is log-singular. Otherwise, by using a pair of two finite length geodesics with the same endpoint on $\Gamma$, we could construct a finite length arc crossing $\Gamma$ at exactly one point. Our example gives the first explicit construction of a log singular curve. For more on log-singular curves and conformal welding, see \cite{BishopFlexible} and \cite{BishopKoebe}. 

We will construct the curve $\Gamma = \bigcap_{n=1}^{\infty} \Gamma_n$ in Theorem \ref{Main} as a nested intersection of axis aligned tubes which we call \textit{plumbings}. Given a plumbing, we will use a Lakes of Wada type of construction to construct a new plumbing which weaves back and forth inside of itself, see Figure \ref{sample}. For a rectifiable curve  $\sigma$ that passes from the interior boundary component to the exterior boundary component of $\Gamma$, we let $l_n = H^1(\Gamma_n \cap \sigma).$ In order for $l_{n+1}$ to be much smaller than $l_n$, $\sigma$ must weave back and forth horizontally so that it can pass through the vertical gaps, as seen in Figure \ref{sample}. We will tune the parameters defining the plumbings $\Gamma_n$ so that $\sigma$ can only pass through a increasingly smaller percentage of the vertical gaps of $\Gamma_n$ as $n \ra \infty$. It will follow that most of the time, the connected components of $\sigma \cap \Gamma_{n+1}$ must pass through those components approximately as a straight line segment, and we will show directly that a straight line segment intersects $\Gamma$ with a set of positive length. These observations combined will allow us to prove Theorem \ref{Main}.

\begin{figure}[!h]
	\centerline{ \includegraphics[height=3in,width=5in]{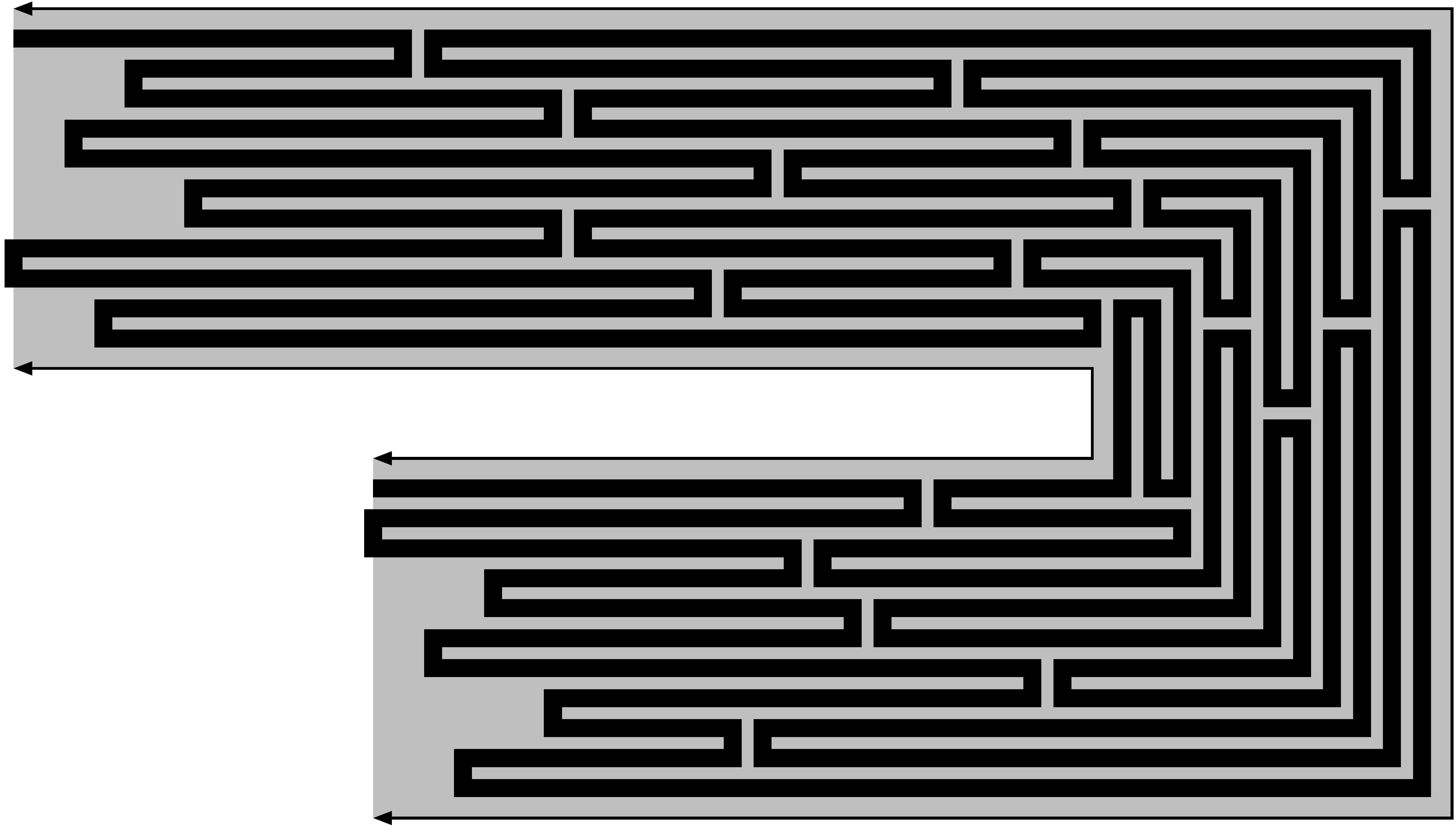} }
	\caption{An example of the procedure taking a plumbing $\Gamma_n$ (pictured in gray) to a new plumbing $\Gamma_{n+1}$ (pictured in black). For a rectifiable curve to intersect $\Gamma$ on a set of zero length, it must intersect $\Gamma_{n}$ on a set of length which tends to $0$ as $n \ra \infty$. This requires $\sigma$ to weave back and forth too much in comparison to the width of the rectangles that define the plumbings, and this will cost a large amount of length if this happens too often.}
	\label{sample}
\end{figure}

The structure of the paper is as follows. In Section 2, we define all the necessary terminology and reduce Theorem $1.1$ to the case of rectifiable curves with length $1$. In Section $3$, we carefully define plumbings and discuss their properties. In Sections 4 and 5, we show how to create new plumbings from old to create the limiting Jordan curve $\Gamma$. In Section 6, we show that the Jordan curve cannot be crossed at a single point by rectifiable curves, and in Section 7 we refine this to show the intersection of the rectifiable curve with the Jordan curve has positive area.

\section{Preliminaries and Generalities}
A \textit{curve} is a continuous function $\sigma:[0,1] \ra \bC$. If $\sigma$ is a curve with $\sigma(0) = \sigma(1)$ we call $\sigma$ \textit{closed}. We will sometimes refer to curves that are not closed as \textit{arcs}. A curve $\sigma$ is called $\textit{simple}$ if for all $t_1 \neq t_2 \in [0,1]$, we have $\sigma(t_1) \neq \sigma(t_2)$.  Similarly, we call a closed curve simple if for all $t_1 \neq t_2 \in [0,1)$, we have $\sigma(t_1) \neq \sigma(t_2)$.  A \textit{Jordan curve} is a simple closed curve. The image $\sigma([0,1]) \subset \bC$ is the \textit{trace} of $\sigma$. When it will not cause confusion, we will refer the trace of a curve $\sigma$ and the curve $\sigma$ interchangeably. Recall that the \textit{length} of a curve $\sigma$ is defined to be
$$l(\sigma) = \sup_{0 =t_0 \leq t_1 \leq \dots \leq t_n =1} \sum_{k=1}^n |\sigma(t_k) - \sigma(t_{k-1})|.$$ 
If $l(\sigma) < \infty$, we call the curve or arc $\sigma$ \textit{rectifiable}.



\begin{definition}
Let $\Gamma \subset \bC$ be a Jordan curve. We say that an arc $\sigma$ \textit{crosses} $\Gamma$ if $\sigma(0)$ and $\sigma(1)$ are in distinct complementary components of $\Gamma$. 
\end{definition}
If $\sigma$ crosses $\Gamma$, we are interested in the subset of points where $\sigma$ passes from one complimentary component of $\Gamma$ to the other.
\begin{definition}
We say that $\sigma$ \textit{pierces $\Gamma$ at $x$} if for every $\epsilon >0$, $B(x,\epsilon) \cap \sigma$ contains points in both complementary components of $\Gamma$. If $\sigma$ crosses $\Gamma$, the \textit{piercing set} of $\sigma$, $\cP(\sigma)$, is the nonempty set of all points $x \in \Gamma$ where $\sigma$ pierces $\Gamma$ at $x$. $\Gamma$ is called \textit{pierceable} by rectifiable curves if there exists a rectifiable curve $\sigma$ crossing $\Gamma$ so that $\cP(\sigma)$ is exactly one point. Otherwise $\Gamma$ is \textit{unpierceable} by rectifiable curves.
\end{definition}

The definition of pierceability can be easily adjusted to include other families of curves, such as line segments, simple rectifiable curves, and rectifiable curves of some given length. If $\Gamma$ is unpierceable by rectifiable curves, the piercing set is nontrivial.

\begin{lemma}
\label{Cantor}
Let  $\Gamma$ be a Jordan curve. The following are equivalent: 
\begin{enumerate}
	\item $\Gamma$ is unpierceable by rectifiable arcs.
	\item $\Gamma$ is unpierceable by simple rectifiable arcs.
	\item $\Gamma$ is unpierceable by simple rectifiable arcs with length $\leq 1$
	\item For any simple rectifiable $\sigma$ crossing $\Gamma$, $\cP(\sigma)$ is uncountable.
	\item For any rectifiable $\sigma$ crossing $\Gamma$, $\cP(\sigma)$ is uncountable.
\end{enumerate} 
\end{lemma}
\begin{proof}
It is easy to see that $(1) \Rightarrow (2) \Rightarrow (3)$ and obviously $(5) \Rightarrow (1)$. $(4) \Rightarrow (5)$ because if $\sigma$ is a rectifiable arc, there exists a simple rectifiable arc $\sigma'$ which is a subset of the trace of $\sigma$ with the same endpoints. So it is sufficient to show that $(3) \Rightarrow (4)$.

First observe that $\cP(\sigma)$ cannot have any isolated points, otherwise there exists a sub-curve $\sigma' \subset \sigma$ in a neighborhood of the isolated point that pierces $\Gamma$ at exactly one point. Since closed countable sets must have isolated points, it will be sufficient to show that $\cP(\sigma)$ is closed.

Let $\Omega_1$ be the bounded complementary component of $\Gamma$, and $\Omega_2$ be the unbounded complementary component of $\Gamma$.  Write $\sigma:[0,1] \ra \bC$. Let $A = \sigma^{-1}(\Omega_1)$ and $B = \sigma^{-1}(\Omega_2)$. $A$ and $B$ are disjoint open subsets of $[0,1]$. We claim that
\begin{equation}
\label{PiercingSet}
\cP(\sigma) = \sigma(\overline{A} \cap \overline{B}).
\end{equation}

If $x \in \cP(\sigma)$, there exists sequences $z_n \subset \sigma \cap \Omega_1$ and $w_n \subset \sigma \cap\Omega_2$ so that $z_n \ra x$ and $w_n \ra x$. Since $\sigma$ is simple, if $\sigma(t) = x$, there are corresponding sequences $t_n$ and $t_n'$ so that $z_n = \sigma(t_n)$, $w_n = \sigma (t_n')$, and both $t_n, t_n' \ra t$. So $t \in \overline{A} \cap \overline{B}$, and therefore $\sigma(t) = x \in \sigma(\overline{A} \cap \overline{B})$.

If $x \notin \cP(\sigma)$, then there exists an $\epsilon > 0$ so that either $B(x,\epsilon) \cap \sigma$ does not contain any points in $\Omega_1$, or $B(x,\epsilon) \cap \sigma$ does not contain any points in $\Omega_2$. If $B(x,\epsilon) \cap \sigma$ does not contain any points in $\Omega_1$, and if $\sigma(t) = x$, then there exists a $\delta > 0$ so that $\sigma (t-\delta,t+\delta) \subset B(x,\epsilon/2)$ which implies that $\sigma(t-\delta,t+\delta) \cap \Omega_1$ is empty. It follows that $t \notin \overline{B}$, and since $\sigma$ is simple it follows that $x \notin \sigma(\overline{A} \cap \overline{B})$. The argument is exactly the same if $B(x,\epsilon) \cap \sigma$ does not contain any points in $\Omega_2$. Therefore, (\ref{PiercingSet}) holds.

Since $\overline{A} \cap \overline{B}$ is compact and $\sigma$ is continuous, $\cP(\sigma)$ is compact and therefore must be closed. This proves the claim. 
\end{proof}


In fact, it is not difficult to see additionally that $\cP(\sigma)$ is a Cantor set: compact, uncountable, totally disconnected, and no isolated points. 

\section{Plumbing}
The curve $\Gamma$ we construct will be the nested intersection of a sequence of topological annuli which we call \textit{plumbings}. We first define the two basic pieces that determine a plumbing. See Figure \ref{PlumbingPieces}.
 
\begin{figure}[!h]
	\centerline{ \includegraphics[height=1in,width=5in]{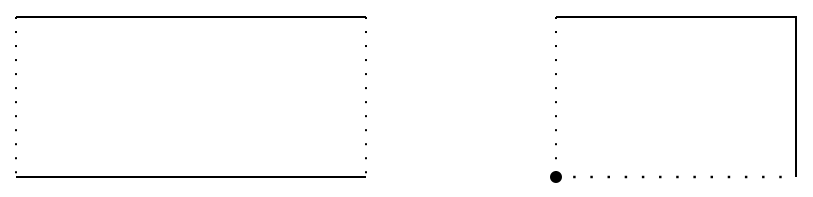} }
	\caption{Plumbings are constructed by alternating straight pieces and corner pieces, and gluing them together by arc length. The boundary sides are filled in and the openings are denoted by the dotted lines.}
	\label{PlumbingPieces}
\end{figure}

To form a \textit{straight piece}, start with a rectangle with sides parallel to the coordinate axes, and remove two of the opposite sides. We call the remaining sides the \textit{boundary sides} of the straight piece and the sides we removed the \textit{openings} of the straight piece. We call the boundary side that is a subset of the interior boundary component of the plumbing the \textit{top} of the straight piece, and the boundary side which is a subset of the exterior boundary component of the plumbing the \textit{bottom} of the straight piece. Whenever we need to analyze a specific straight piece, we will orient our point of view to justify these names. 

\textit{Corner pieces} are formed by taking an axis aligned rectangle and removing two adjacent sides, but keeping the mutual vertex of the two removed sides. This vertex is called an \textit{inner corner} for the corner piece, and the remaining sides are similarly called \textit{boundary sides} and \textit{openings}.  Plumbings are formed by alternating straight and corner pieces and gluing together their openings by arc length. See Figure \ref{Example} for an example of a plumbing.
\begin{definition}
A \textit{plumbing} $P$ is a closed topological annulus obtained by gluing together straight pieces and corner pieces so that between every two corner pieces there exists exactly one straight piece. 
\end{definition}

\begin{figure}[!h]
	\centerline{ \includegraphics[height=2in,width=5in]{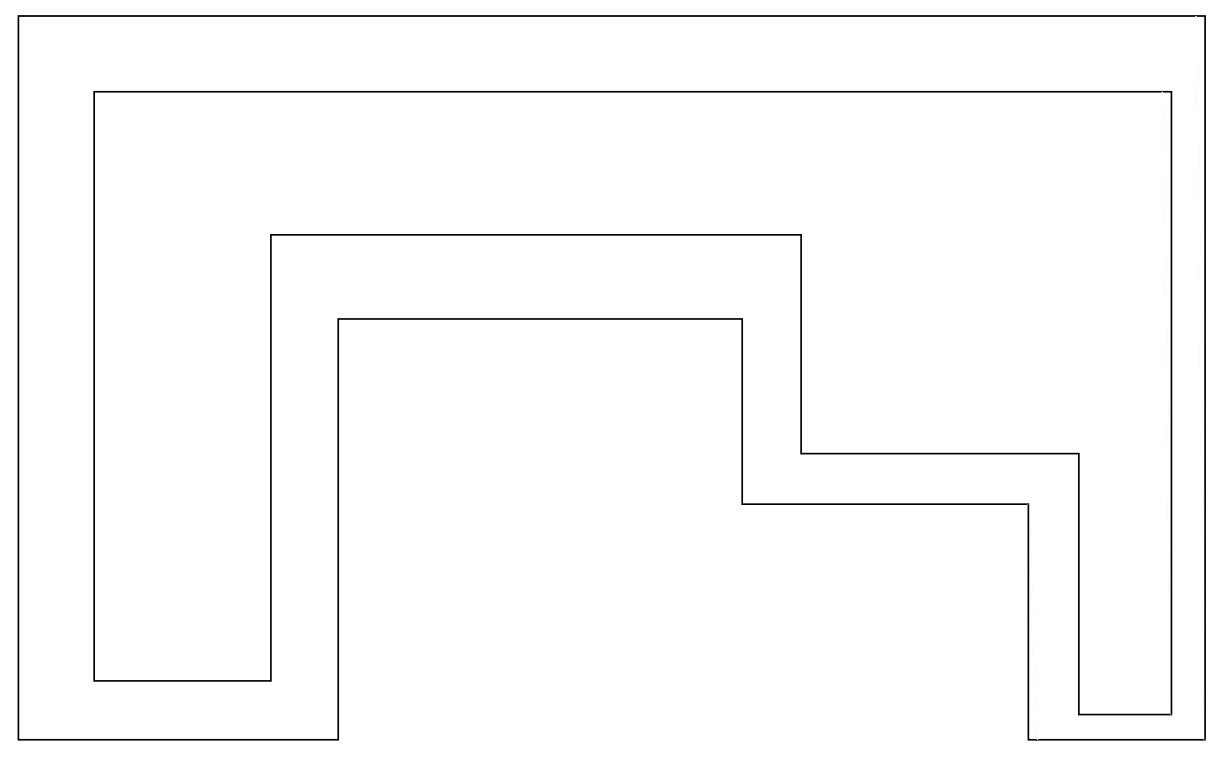} }
	\caption{An example of plumbing.}
	\label{Example}
\end{figure}

We define the following geometric quantities associated to a plumbing. The \textit{width} of a straight piece is defined to be the distance between its boundary sides. The \textit{width} of a plumbing is defined to be the maximum of the widths of the straight pieces in the plumbing. The \textit{length} of a straight piece is the distance between its openings. The \textit{min-length} of a plumbing is the minimum of the lengths of its straight pieces.

\begin{figure}[!h]
	\centerline{ \includegraphics[height=2in,width=5in]{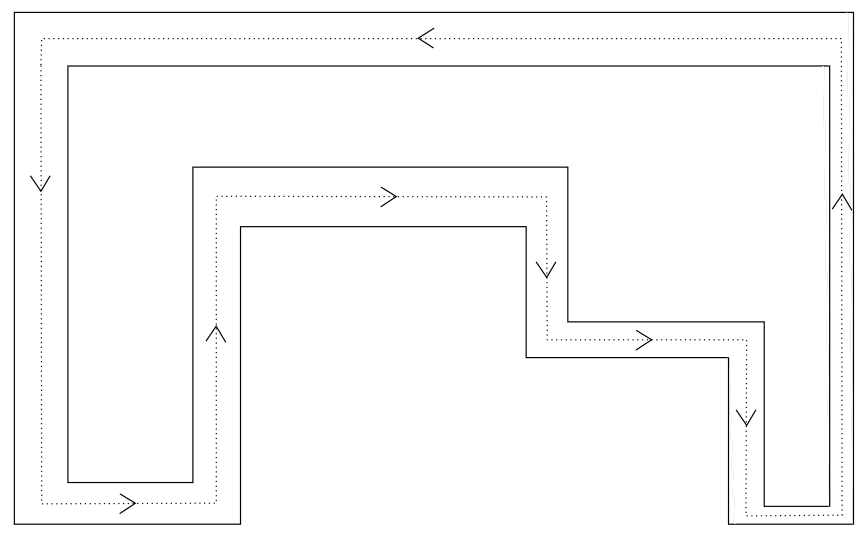} }
	\caption{A plumbing with its core curve in the middle, oriented counterclockwise.}
	\label{foliation}
\end{figure}

Plumbings come with a natural foliation of axis aligned polygonal curves that follow the inner and outermost boundary. Start with a plumbing $P$, and let $\gamma_0$ and $\gamma_1$ denote the outermost and innermost boundary components, respectively. Decompose $P$ into its straight and corner pieces. Given a straight piece $S$ and a number $t \in (0,1)$, let $w$ denote its width. Let $s_t$ denote the line segment parallel to the boundary sides of $S$ with distance $t \cdot w$ from $\gamma_0$. Draw this line segment for all straight pieces in $P$. Any corner piece is adjacent to two straight pieces, and each of its openings contains an endpoint of some $s_t$ and some $s_t'$. Continue the segments $s_t$ and $s_t'$ until they intersect (this point of intersection will be on the diagonal that separates the openings of the corner piece) and repeat this for all corner pieces to form a Jordan curve $\gamma_t$. We define $\gamma_{1/2}$ to be the \textit{core curve} of $P$. We call $\{\gamma_t\}$ the \textit{plumbing foliation} of $P$. Whenever we parameterize the elements of a plumbing foliation, we do so counterclockwise, so that the inner boundary component is always to the left of the curve in the foliation. See Figure \ref{foliation}.

\begin{figure}[!h]
	\centerline{ \includegraphics[height=2in,width=5in]{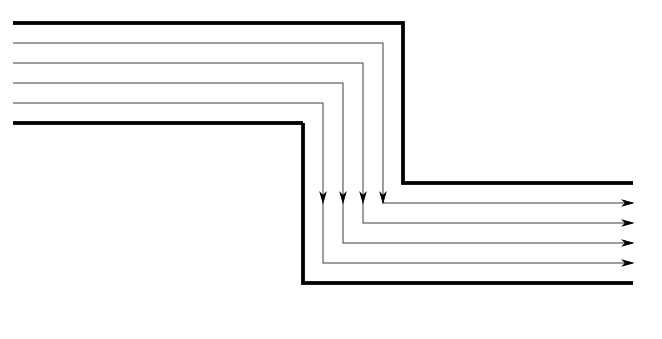} }
	\caption{An illustration of some elements of the plumbing foliation, with orientation.}
	\label{foliationex}
\end{figure}
.

A point on the core curve of $P$ where the core curve changes from a vertical segment to a horizontal segment or vice versa is called a $\textit{bend point}$. Two bend points are called \textit{adjacent} if there is no bend point between them. Every bend point is adjacent to exactly two other bend points. Let $z$ and $w$ be two adjacent bend points of $P$, so that the core curve passes through $z$ first. The \textit{$(z,w)$-junction} is the union of the two corner pieces containing $z$ and $w$, and the straight piece that is between $z$ and $w$. Each $(z,w)$ junction takes two possible forms, depending on how the openings of the corner pieces are configured. We call the junctions $U$-junctions and $Z$-junctions, respectively. When considered with respect to the orientation of the core curve, there are four junctions to consider. See Figure \ref{Junctions}.
\begin{figure}[!h]
	\centerline{ \includegraphics[height=5in,width=5in]{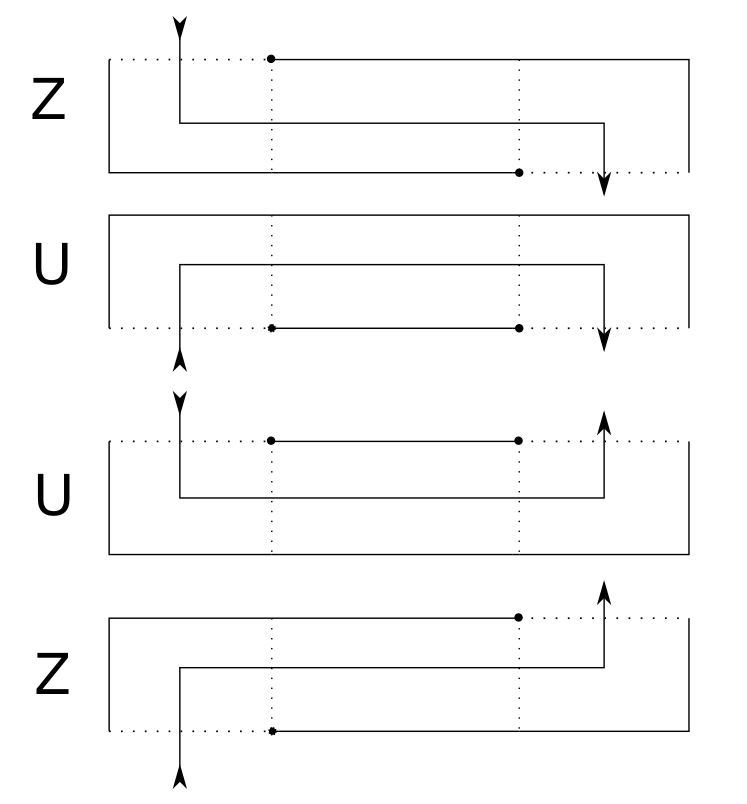} }
	\caption{Accounting for orientation, there are two Z-junctions and two U-junctions. The pictures have been oriented so that the top boundary side of the straight piece is a subset of the innermost boundary of the plumbing. In fact, since a plumbing is axis aligned, we can always orient our view of a straight piece or junction in this way.}
	\label{Junctions}
\end{figure}

\section{Creating New Plumbing from Old}
In this section we define a procedure which takes a plumbing $\Gamma_n$ and creates a plumbing $\Gamma_{n+1}$ contained inside of $\Gamma_n$ taking up most of the area of $\Gamma_n$, but having a much smaller width. Roughly put, we will insert very thin vertical and horizontal rectangles segments inside of $\Gamma_n$ to form a new plumbing. This procedure will form the basis of our construction.

Let $\Gamma_n$ be a plumbing. Define 
\begin{equation}
d_n := \min\{l_n, w_n\}
\label{dn}
\end{equation}
where $l_n$ is the min-length of $\Gamma_n$ and $w_n$ is the width of $\Gamma_n$. Subdivide $\Gamma_n$ into straight pieces and corner pieces, and denote the plumbing foliation of $\Gamma_n$ by $\{\gamma^n_t\}$. Parameterize $\gamma^n_{1/2}$ as a Jordan curve, say, by arc length. Denote $P_0$ as the unique plumbing piece of $\Gamma_n$ so that $\gamma^n_{1/2}(0) \in P_0$ and so that there exists $\delta >0$ so that $\gamma^n_{1/2}:[0,\delta) \ra P_0$. Label the rest of the plumbing pieces $\{P_i\}_{i=0}^{m-1}$ according to the order $\gamma^n_{1/2}(t)$ passes through them. We will consider this list of pieces modulo $m$, so that $P_m = P_0$. 

Choose some $\delta_n \in (0,1/100)$. We are going to subdivide the straight pieces $P_i$ further into rectangles by connecting the boundary sides by line segments placed distance approximately $\delta_n d_n$ apart; no line segment will be farther than $2\delta_n d_n$ from the adjacent line segments we place, and no line segment will be closer than $1/2 \delta_n d_n$ from the adjacent line segments we place. We will simply refer to these new rectangles as \textit{subdivided rectangles}. If $P_i$ is the straight piece of a $U$-junction, we must make sure that the number of subdivide rectangles of $P_i$ is even. If $P_i$ is an element of a $Z$-junction, the number of subdivided rectangles should be odd. See Figure \ref{combinatorics}.

\begin{figure}[!h]
	\centerline{ \includegraphics[height=2in,width=5in]{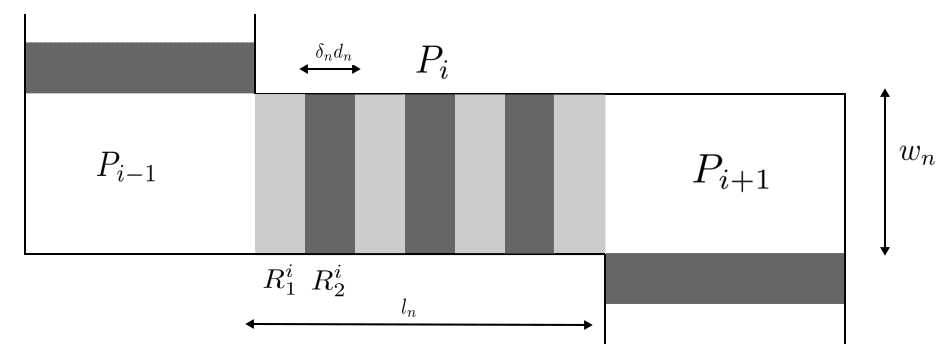} }
	\caption{Illustrated above is a $Z$-junction, with pieces $P_{i-1},P_i$ and $P_{i+1}$. Since the straight piece is in a $Z$-junction, we subdivide $P_i$ into an odd number of rectangles. Since the boundary sides of $P_{i-1}$ extend to the bottom boundary side of $P_i$, the first subdivided rectangle will be labeled as `T'.  The light gray color coincides with rectangles labeled as `T' and the dark gray color coincides with rectangles labeled as `B'. In practice, $\delta_n$ is chosen so that $P_i$ is broken into at least 100 rectangles.}
	\label{combinatorics}
\end{figure}

Given any $t_n \in \bN$, we select the elements of the plumbing foliation $\{\gamma^n_{j/t_n}\}_{j=0}^{t_n}.$ We will always assume that $t_n$ is an odd number, so that we are selecting $t_n + 1$ equally spaced elements of the foliation. If $w_n$ is the width of $\Gamma_n$, then adjacent elements of the plumbing foliation inside of a straight piece are no more than distance $\frac{w_n}{t_n}$ apart.



For each straight piece $P_i$, we want to assign a labels $T$ and $B$ for its subdivided rectangles.
\begin{enumerate}
	\item If $P_{i-1}$'s boundary side extends the bottom boundary side of $P_i$, the first subdivided rectangle is marked as T.
	\item If $P_{i-1}$'s boundary side extends the upper boundary side of $P_i$, the first subdivided rectangle is marked as B.
\end{enumerate}
We alternate T and B until we reach the end of the straight piece.

Given a straight piece $P$, we have $P = \cup R_i$, where $R_i$ are the subdivided rectangles of $P$ described above with base approximately $\delta_n d_n$ and with height no greater than $w_n$. Choose some positive integer $v_n = k^2$ for some $k \geq 3$ and then choose some subdivided rectangle $R_i$. Decompose the base of $R_i$ into $v_n$ many equal length segments. We will need the following definition and combinatorial fact.

\begin{definition}
Split the unit square a $v_n$ by $v_n$ square grid. We define a \textit{rook placement} to be a selection of exactly $v_n$ many squares, which we call \textit{rooks}, in the grid so that:
\begin{enumerate}
	\item Each row in the grid contains exactly one rook.
	\item Each column in the grid contains exactly one rook.
\end{enumerate}
\end{definition}

The following lemma is simple, and we leave its proof to the reader, with Figure \ref{RookPlacement} as a hint. A rook placement satisfying the conditions of Lemma \ref{rook} is called a \textit{good $v_n$-rook placement}.

\begin{lemma}
	\label{rook}
Let $v_n = k^2$ for some integer $k \geq 3$, and suppose the unit square has been decomposed into a $v_n$ by $v_n$ square grid. Identify the sides of the unit square to create a torus. Then there exists a rook placement of the unit square so that for every rook $Q$, the $8$ adjacent squares, viewed on the torus, do not contain any other rooks in the rook placement.
\end{lemma}

\begin{figure}[!h]
	\centerline{ \includegraphics[height=3in,width=3in]{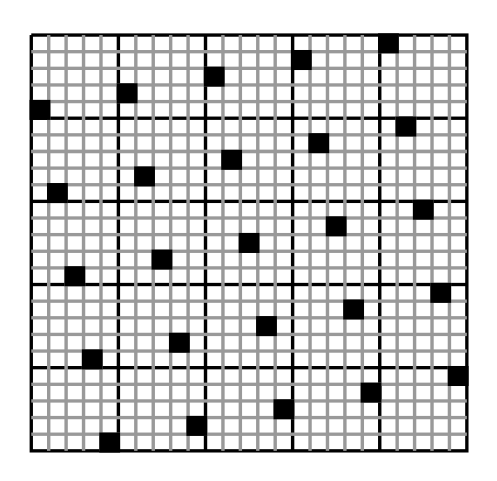} }
	\caption{A good $v_n$-rook placement for $v_n = 5^2$.}
	\label{RookPlacement}
\end{figure}

Supose that $R_i$ is a subdivided rectangle of some straight piece, and that $R_i$ is designated as T. The plumbing foliation cuts $R_i$ into $t_n$ many rectangles. Orienting our point of view and counting downwards from the boundary side intersecting the interior complementary component, we call the first $2v_n$ many rectangles a \textit{$v_n$-slab} inside $R_i$.  A $v_n$-slab decomposes into a $v_n$ by $v_n$ grid. Rows are determined by every two elements of the plumbing foliation, and the columns are determined by the $v_n$ equally spaced points on the base of $R_i$.  We choose a good $v_n$-rook placement for this grid. For each rook in this good $v_n$-rook placement, place a vertical line segment connecting the top and bottom of the rook through the middle. This segment has endpoints on two elements of the plumbing foliation for $\Gamma_n$ that determine the rows, and passes through exactly one foliation element between the two. See Figure \ref{vnblock}.

\begin{figure}[!h]
	\centerline{ \includegraphics[height=4in,width=6in]{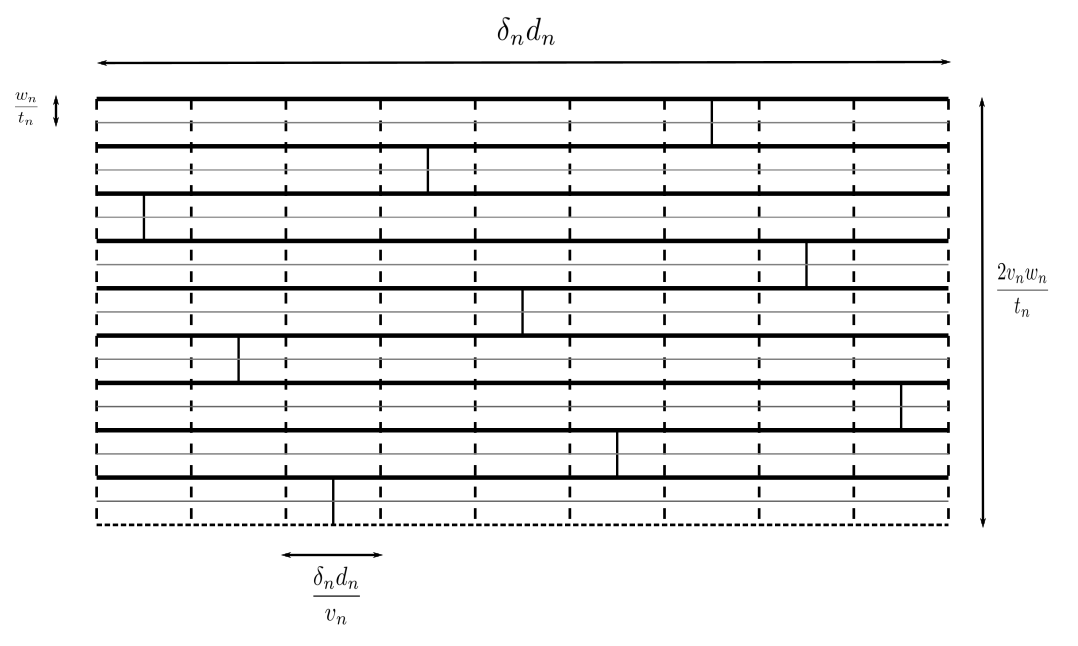} }
	\caption{A good $v_n$-rook placement for a $v_n$-slab. Each combinatorial square has base $\delta_n d_n/v_n$ and height $2w_n/t_n$. The $v_n$-slab the rook placement is inside has width $\delta_n d_n$ and height $2v_nw_n/t_n$. The foliation elements that determine the rows of the $v_n$ slab are colored darker than the ones that are skipped and that the vertical segments intersect.}
	\label{vnblock}
\end{figure}

We can split the rest of $R_i$ into $v_n$ slabs, going from the top down and repeating the rook placement of the top $v_n$-slab until we reach the element $\gamma^n_{1/t_n}$ of the plumbing foliation. We can choose $t_n$ to be larger if we need to so that the final $v_n$-slab has a bottom side determined by $\gamma^n_{1/t_n}$; we will always assume that $t_n$ satisfies this property.

One $v_n$ slab forms a rectangle with base length approximately $\delta_n d_n$ and height $2v_n w_n/t_n$. In our construction, the height $2v_nw_n/t_n$ will be much smaller than $\delta_n d_n$. The amount of $v_n$ slabs we need to stack to form an approximate $\delta_n d_n$ by $\delta_n d_n$ \textit{$v_n$-square} is approximately $k_n$, where,
$$\frac{2v_n w_n}{t_n} \cdot k_n = \delta_n d_n.$$
Or,
\begin{equation}
\label{kn}
k_n = \frac{\delta_n d_n t_n}{2 v_n w_n}.
\end{equation}
We also can see that $R_i$ decomposes into approximately
\begin{equation}
\label{lambdan}
\lambda_n = \frac{w_n}{\delta_n d_n}
\end{equation}
many $v_n$-squares.  

We do the same thing with the rectangles labeled $B$, except instead of building $v_n$-slabs starting from the top down, we build $v_n$-slabs starting from the bottom up. We do not place any segments connecting elements of the plumbing foliation in corner pieces. 

We now draw line segments along the plumbing foliation. In corner pieces, we add the intersection of all elements in the plumbing foliation with the corner piece. See Figure \ref{CornerFoliation}. In straight pieces, we do the same thing, except we remove the portion of the plumbing foliation within distance $w_n/t_n$ of the center of vertical line segments we placed in the previous step. See Figure \ref{RookSlab}. 

\begin{figure}[!h]
	\centerline{ \includegraphics[height=3in,width=4in]{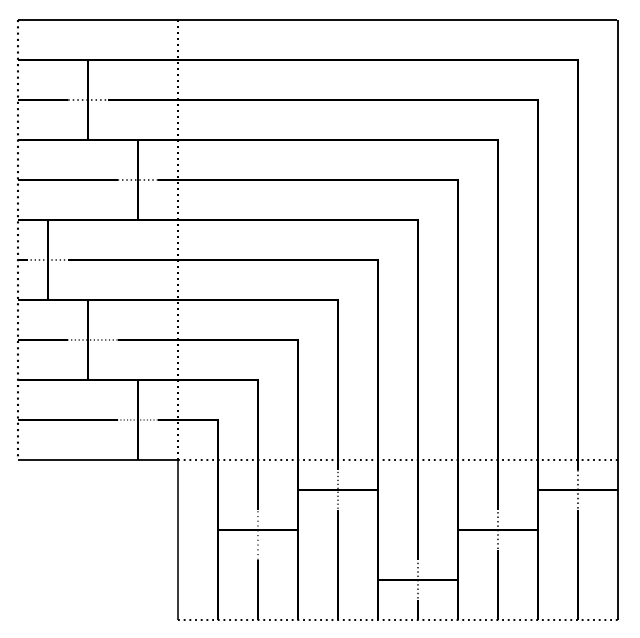} }
	\caption{In the corner pieces of the plumbing, we just fill in the elements of the plumbing foliation.}
	\label{CornerFoliation}
\end{figure}


With the procedure above of adding in vertical and horizontal line segments into straight and corner pieces, we have constructed a topological subannulus contained inside of $\Gamma_n$. The only issue remaining is that its boundary components are not Jordan curves. To do this, we will thicken all of the segments in the horizontal and vertical directions by a small amount so that the segments we added become rectangles. The amount we thicken is given by
\begin{equation}
\label{etan}
\eta_n := \frac{1}{100^n}\frac{w_n}{\prod_{j=1}^n t_j}.
\end{equation}
The following theorem is now clear:

\begin{theorem}
$\Gamma_{n+1}$ is a plumbing compactly contained in $\Gamma_n$.
\end{theorem}

The core curve of $\Gamma_{n+1}$ can be visualized in Figure \ref{NewCoreCurve}.

\begin{figure}[!h]
	\centerline{ \includegraphics[height=4in,width=6in]{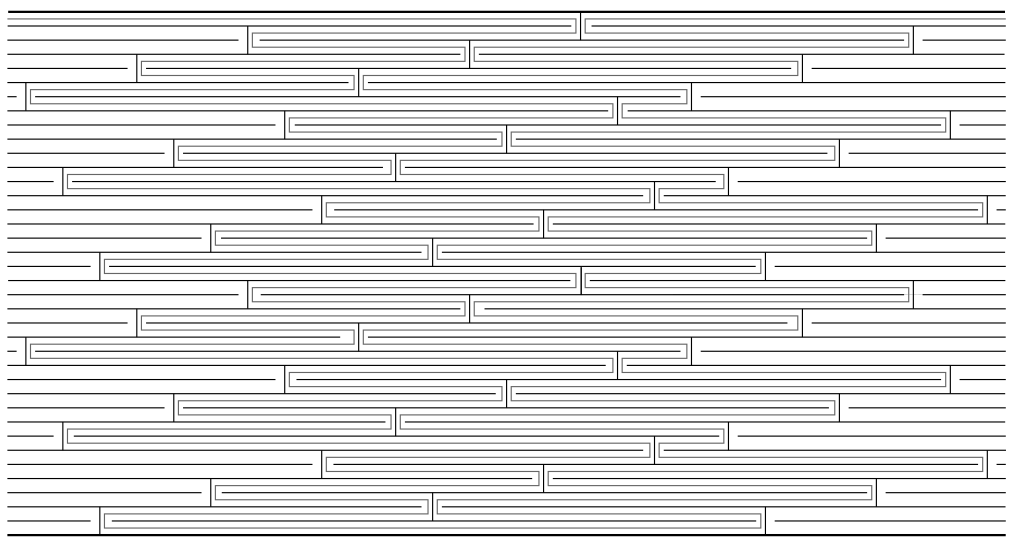} }
	\caption{A visualization of $\Gamma_{n+1}$ inside of a straight piece, with its core curve in gray. With this point of view, the core curve enters from the top left hand corner, weaves back and forth between vertical segments we place down to the bottom, and then weaves back up. This picture was constructed with one $v_n$-square made out of $2$ $v_n$-slabs. In practice, the number of $v_n$-slabs and $v_n$-squares tends to infinity.}
	\label{NewCoreCurve}
\end{figure}

\section{The Limiting Object is a Jordan curve}
We now construct the Jordan curve $\Gamma$. Fix $\epsilon_0 > 0$ and let $\Gamma_0$ be the topological annulus formed by taking the open square of side length $2(1+\epsilon_0)$ centered at the origin and removing the closed square of side length $2(1-\epsilon_0)$ centered at the origin. $\Gamma_0$ is a square plumbing, and parameterize its core curve using a constant speed parameterization. Using the procedure in the previous section, we construct a plumbing with the following parameters. At each stage $n$, we must choose a parameter $v_n$. We assume that $v_n$ is strictly increasing, and we always choose $v_n$ to be a perfect square. Moreover, we will demand that
\begin{equation}
\label{vn}
\sum_{n=1}^{\infty} \frac{1}{v_n} < \frac{1}{100}.
\end{equation}
We will always assume that the parameters $\delta_n$ satisfy
\begin{equation}
\label{deltan}
\sum_{n=1}^{\infty} \delta_n < \frac{1}{100}.
\end{equation}

Notice that for any valid choice of $t_n$, the widths $w_n$ will always satisfy the inequality
\begin{equation}
\label{wntn}
w_{n+1} \leq \frac{w_n}{t_n}.
\end{equation}
Given these above quantities, we will always choose $t_n$ large enough so that it satisfies
\begin{equation}
\label{biggerthanone}
\frac{t_n}{v_n} \cdot \frac{(\delta_n d_n)^2}{2w_n} \geq 1
\end{equation}
By (\ref{vn}) and (\ref{biggerthanone}) we know that 
\begin{equation}
\label{tn}
\sum_{n=1}^{\infty} \frac{1}{t_n} < \frac{1}{100}.
\end{equation}
From this it easily follows that 
\begin{equation}
\label{wnconverges}
\sum_{n=1}^{\infty} w_n < \infty.
\end{equation}

In this section, we will prove the following:

\begin{theorem}
\label{JordanCurve}
With all the parameters defined as above, 
$$\Gamma = \bigcap_{n=0}^{\infty} \Gamma_n$$
is a Jordan curve. 
\end{theorem}

To prove this, the first step is to construct appropriate parameterizations of $\gamma^{n+1}_{1/2}$ given some parameterization of the core curves $\gamma^n_{1/2}$ of $\Gamma_n$. Since we will always work with the core curves in this section, we will call $\gamma^n_{1/2} := \gamma^n$.  

Recall that each straight piece $P$ of $\Gamma_n$ is decomposed into subdivided rectangles with base approximately $\delta_n d_n$ and height $w_n$. We do not decompose the corner pieces and leave them as is. For some small choice of $\delta$, $\gamma^{n}(0,\delta)$ is contained in one of these subdivided rectangles or corners, so we denote it as $R_0$. We label the rest of the subdivided rectangles and corner pieces in the order that $\gamma^n$ passes through them. This gives a list of subdivided rectangles and corners $\{R_i\}_{i=0}^{m-1}$. We will consider this list modulo $m$, so that $R_0 = R_m$. Let and $t_i \in [0,1]$ for $i =1,\dots m$ be the first entry time of $\gamma^n$ in $R_i$. $t_{m}$ will be the exit time for $R_{m-1}$, which coincides with entering back inside $R_0$. $\gamma^n:(t_{m},1] \cup [0,t_1)$ goes from the end of $R_{m-1}$ to the end of $R_0$. 

We'll show how to parameterize $\gamma^{n+1}$ inside of the intervals $(t_i,t_{i+1})$. Here we take the convention that $(t_{m},t_{m+1}) = (t_{m},1] \cup [0,t_1)$. To do this, we will show where to place the points $\gamma^{n+1}(t_i)$ and $\gamma^{n+1}(t_{i+1})$. Then for the subarc of $\gamma^{n+1}$ in between these points, we will use a constant speed parameterization. We will handle this with three cases. 

For the first case, suppose that $R_i$ is a subdivided rectangle of a straight piece so that $R_{i+1}$ is not a corner piece. Orient $R_i$ so that the top boundary side is a subset of the innermost boundary component of $\Gamma_n$. If $R_i$ was labeled as B, then $\gamma^{n+1}(t_i)$ will be the point on the intersection of $\gamma^{n+1}$ and the left side of $R_i$ between $\gamma^n_{\frac{t_n-1}{t_n}}$ and $\gamma^n_1$. We let $\gamma^{n+1}_{1/2}(t_{i+1})$ be the point on the intersection of $\gamma^{n+1}$ and the right side of $R_i$ between $\gamma^n_{\frac{1}{t_n}}$ and $\gamma^n_0$. If $R_i$ was labeled as T, we do the opposite. $\gamma^{n+1}(t_i)$ will be the point on the intersection of $\gamma^{n+1}$ and the left side of $R_i$ between and $\gamma^{n}_{\frac{1}{t_n}}$ and $\gamma^n_0$, and $\gamma^{n+1}(t_{i+1})$ will be be the point on the intersection of $\gamma^{n+1}$ and the left side of $R_i$ between $\gamma^{n}_\frac{t_n-1}{t_n}$ and $\gamma^n_1$. See Figure \ref{StraightParameterization}.

Suppose that $R_i$ is a subdivided rectangle so that $R_{i+1}$ is a corner boundary piece. Then we label $\gamma^{n+1}(t_i)$ using the procedure above. We let $\gamma^{n+1}(t_{i+1})$ be the point in the same straight piece as $\gamma^{n+1}(t_i)$ that lies on the right side of $R_i$. See Figure \ref{CornerParameterization}.

Finally, suppose that $R_i$ is a corner piece. Then $R_{i-1}$ and $R_{i+1}$ fall into the above cases, which means that $\gamma^{n+1}(t_i)$ and $\gamma^{n+1}(t_{i+1})$ have already been defined. The former corresponds to the first time that $\gamma^{n+1}$ enters the corner piece, and the latter corresponds to the last time $\gamma^{n+1}$ exits the corner piece. We again use the constant speed parameterization to parameterize the subarc between these two points. See Figure \ref{CornerParameterization}.

\begin{figure}[!h]
	\centerline{ \includegraphics[height=3in,width=5in]{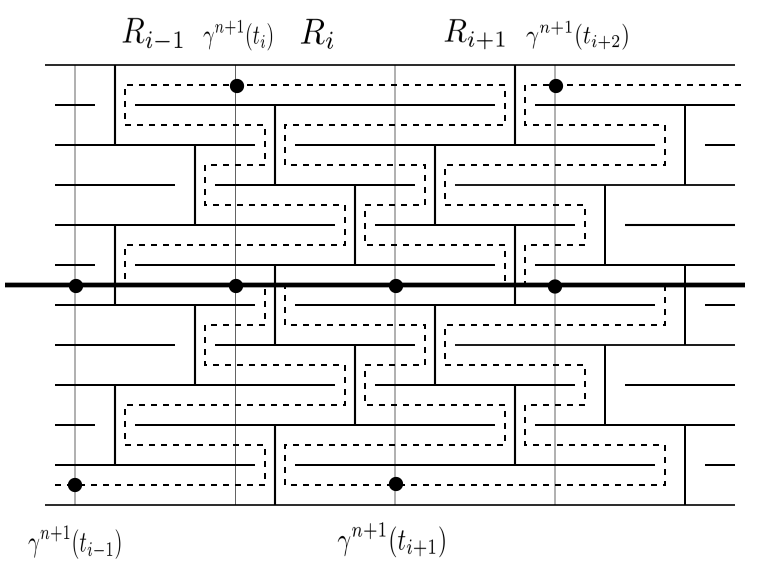} }
	\caption{Parameterizing in straight pieces. Denote the subdivided rectangles by $R_{i-1}, R_i, R_{i+1}$. The bold curve is $\gamma^n$, and $\gamma^n(t_j)$ is marked for $j = i-1,i,i+1,i+2$. Since $R_{i-1}$ is a T subdivided rectangle, $\gamma^{n+1}(t_{i-1})$ is the point on $\gamma^{n+1}$ with farthest distance below $\gamma^n(t_{i-1})$. Since $R_i$ is a B subdivided rectangle, $\gamma^{n+1}(t_i)$ is the point on $\gamma^{n+1}$ with farthest distance above $\gamma^n(t_i)$. We parameterize the portion of the curve between those points with a constant speed parameterization.}
	\label{StraightParameterization}
\end{figure}

\begin{figure}[!h]
	\centerline{ \includegraphics[height=3in,width=4in]{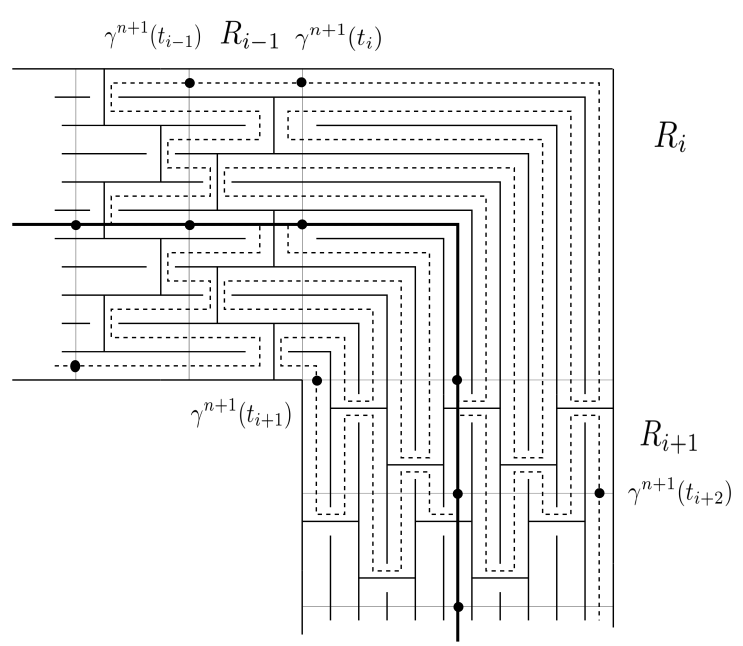} }
	\caption{Parameterizing near a corner piece. Let $R_{i-1}$ be the subdivided rectangle to the left of the corner piece, $R_{i}$. $\gamma^{n+1}:(t_{i-1},t_{i})$ is simply the portion of $\gamma^{n}(t_{i-1},t_{i})$ translated up to be between the final two elements of the plumbing foliation. $\gamma^{n+1}(t_{i},t_{i+1})$ is the constant speed parameterization between the already placed $\gamma^{n+1}(t_{i})$ and $\gamma^{n+1}(t_{i+1})$.}
	\label{CornerParameterization}
\end{figure}

\begin{lemma}
The sequence of functions $\{\gamma^{n}\}_{n=1}^{\infty}$ is uniformly Cauchy. In fact, for sufficiently large $n$, we have the estimate for all $t \in [0,1]$
$$|\gamma^{n+1}(t) - \gamma^n(t)| \leq 5 w_n.$$
\end{lemma}
\begin{proof}
If $\gamma^n(t) \in R_i$, and $t \in [t_i, t_{i+1}]$, observe that the curve $\gamma^{n+1}_{1/2}: [t_i,t_{i+1}]$ can only intersect the rectangles $R_{i-1}, R_i$ and $R_{i+1}$. Since $\diam(R_i) \leq \sqrt{2}w_n$ for all corner pieces and subdivided rectangles, this means that $\diam(\gamma^{n+1}_{1/2}(t_i,t_{i+1})) \leq 3\sqrt{2} w_n \leq 5w_n$. Since $\gamma^n(t) \in R_i$, it follows that 
$$|\gamma^{n+1}(t) - \gamma^{n}(t)| \leq 5 w_n.$$
By $(\ref{wnconverges})$, this is a summable estimate independent of $t \in [0,1]$, so it follows that $\gamma^n$ is uniformly Cauchy.
\end{proof}

It follows that $\gamma^n$ converges uniformly to  $\gamma$ for some continuous function $\gamma: [0,1] \ra \bC$. Since $\gamma$ is continuous and the widths of $\Gamma_n$ are strictly decreasing, we must have $\gamma([0,1]) = \Gamma$. So to show $\Gamma$ is a Jordan curve it is sufficient to show that $\gamma$ is injective.

\begin{lemma}
\label{cantleavehome}
For all sufficiently large $n$, if $\gamma^n(t) \in R_i$ for some $i$, then $\gamma(t) \in R_{i-2}, \dots , R_{i+2}$. 
\end{lemma}
\begin{proof}
We know that $\gamma^{n+1}(t) \in R_{i-1}, R_i$, or, $R_{i+1}$. Then we can estimate using $(\ref{wntn})$ that
\begin{eqnarray*}
|\gamma(t) - \gamma^{n+1}(t)| &\leq& \sum_{j=n+1}^{\infty} |\gamma^j(t) - \gamma^{j+1}(t)| \\
	&\leq&	\sum_{j=n+1}^{\infty} 5 w_j \\
	&\leq& 5 w_{n+1}\left(1+ \sum_{j=n+1}^{\infty} \frac{1}{\prod_{k=n+1}^j t_k} \right) \\
	&\leq& 6 w_{n+1}
\end{eqnarray*}
Next observe that by $(\ref{wntn})$ and $(\ref{biggerthanone})$,
$$w_{n+1} \leq \frac{w_n}{t_n} \leq \frac{(\delta_nd_n)^2}{2v_n} \leq \frac{1}{100} \delta_n d_n.$$
From this and the construction of $\Gamma$ we deduce that $\gamma(t) \in R_{i-2},\dots, R_{i+2}$.  
\end{proof}

\begin{lemma}
\label{cornersgrow}
There exists $n$ large enough so that if $s \neq t$, $\gamma^n_{1/2}(s)$ and $\gamma^n_{1/2}(t)$ are separated by at least $1$ straight piece.
\end{lemma}
\begin{proof}
Suppose that $\gamma^n(s)$ and $\gamma^n(t)$ belong to the same straight piece of $\Gamma_n$. What happens if $\gamma^{n+1}(s)$ and $\gamma^{n+1}(t)$ belong to the same straight piece of $\Gamma_{n+1}$? Then we must have 
\begin{equation}
\label{longer}
|\gamma^n(s) - \gamma^n(t)| \leq |\gamma^{n+1}(s) - \gamma^{n+1}(t)|
\end{equation}
This is because of how we parameterize $\gamma^{n+1}$ with respect to $\gamma^{n}$. Indeed, the procedure we used to parameterize the $\gamma^n$'s implies that the length of $\gamma^{n+1}$ restricted to $[t_i, t_{i+1}]$ is always bounded below by the length of $\gamma^n$ restricted to $[t_i,t_{i+1}]$. Since we use the constant speed parameterization to define $\gamma^{n+1}$ on $[t_i,t_{i+1}]$, we must have inequality $(\ref{longer})$. 

Note that in $\Gamma_{m+1}$ for any $m$, the longest a straight piece can be is bounded above by $10 w_m$. This follows from the construction and Lemma \ref{cantleavehome}. Combined with $(\ref{longer})$, it follows that $\gamma^m(s)$ and $\gamma^m(t)$ cannot belong to the same straight piece for all $m \geq n$. 
\end{proof}

\begin{lemma}
The limit function $\gamma$ above is injective.
\end{lemma}
\begin{proof}
Choose some $s< t$. By Lemma \ref{cornersgrow}, there exists a value $n$ so that $\gamma^n(s)$ and $\gamma^n(t)$ are separated by at least $1$ straight piece. Since $\delta_n < \frac{1}{100}$, there are at least $50$ subdivided rectangles between $\gamma^n(s)$ and $\gamma^n(t)$. Therefore, Lemma \ref{cantleavehome} says that $\gamma(s) \neq \gamma(t)$.
\end{proof}

\begin{corollary}
$\Gamma$ is a Jordan curve.
\end{corollary}

\section{The Jordan curve is non-pierce-able by rectifiable arcs}
By Lemma \ref{Cantor}, to show that $\Gamma$ is non-pierceable by rectifiable arcs, we only must prove that $\Gamma$ is not pierceable by unit length simple rectifiable curves which cross $\Gamma$. 

\begin{lemma}
\label{MustCross}
Suppose that $\sigma$ is a rectifiable arc which crosses $\Gamma$. Then there exists $n$ so that the endpoints of $\sigma$ are in distinct complementary components of $\Gamma_n$. 
\end{lemma}
\begin{proof}
This follows from the fact that $(\ref{dn})$ tends to $0$ as $n \ra \infty$.
\end{proof}

We will always parameterize $\sigma$ crossing $\Gamma$ so that $\sigma(0)$ is in the bounded complementary component of $\Gamma$. The following lemma is easy to visualize, but its proof is cumbersome. The idea is that if $\sigma$ has endpoints in distinct complementary components of $\Gamma_n$, to cross $\Gamma$ at exactly one point, $\sigma$ is forced to remain in the part of $\Gamma_n$ we removed to construct $\Gamma_{n+1}$, otherwise it will cross more than once. Therefore, such a $\sigma$ must remain inside of one subdivided rectangle of $\Gamma_n$, so that it must completely cross a $v_n$-slab. See Figure \ref{StraightParameterization} and Figure \ref{RookSlab}.

\begin{lemma}
\label{crossingslabs}
If $\sigma$ is a rectifiable arc which crosses $\Gamma$ and $\cP(\sigma)$ has only one point, then $\sigma$ must enter every rook contained in some $v_n$-slab of some subdivided rectangle of a straight piece of $\Gamma_n$. Moreover, such a $v_n$-slab may be chosen so that it does not intersect the boundary of $\Gamma_n$.
\end{lemma}
\begin{proof}
By equation $(\ref{vn})$, combined with equation $(\ref{biggerthanone})$, we see that the number of $v_n$-slabs contained in a subdivided rectangle in $\Gamma_n$ tends to $\infty$ as $n \ra \infty$. Therefore, we may always assume that every subdivided rectangle in $\Gamma_n$ has at least one $v_n$-slab in each complementary component of the core curve of $\Gamma_n$ that does not intersect the boundary of $\Gamma_n$.

Suppose that the lemma is false: there is no subdivided rectangle so that $\sigma$ passes through every single rook of some $v_n$-slab contained in that subdivided rectangle. Then there exists a subdivided rectangle $R$ and a $v_n$-slab $S$ contained in $R$ so that $\sigma$ passes through at least one, but not all, of the rooks in $S$. We may assume that $S$ is in the bounded complementary component of the core curve of $\Gamma_n$. Indeed, if no such $v_n$-slab existed, then either $\cP(\sigma)$ has more than one element, or the conclusion of the lemma holds.

In this case, two of the opposite sides of a rook of $S$ that $\sigma$ does not enter are determined by elements of the plumbing foliation $\gamma^n_{j/t_n}$ and $\gamma^n_{(j+2)/t_n}$. Then by the construction of $\Gamma_{n+1}$ from $\Gamma_n$, $\sigma$ must pierce $\Gamma$ at least once between $\gamma^n_{j/t_n}$ and $\gamma^n_{(j+2)/t_n}$. This must happen either in the subdivided rectangle $R$, or one of the subdivided rectangles or corner pieces adjacent to it.

For the exact same reasons as above, there must exist a subdivided rectangle $R'$ and a $v_n$-slab $S'$ contained in $R'$ so that $\sigma$ passes through at least one, but not all of the rooks in the $v_n$-slab, and this $v_n$-slab is in the unbounded complementary component of the core curve of $\Gamma_n$. But then by the same reasoning $\sigma$ must pierce $\Gamma$ again, contradicting the fact that $\cP(\sigma)$ has only one point.
\end{proof}

The statement and proof of Lemma \ref{crossingslabs} apply in exactly the same way if $v_n$-slabs are replaced by $v_n$-squares, since equation $(\ref{deltan})$ implies that $\lambda_n$, the number of $v_n$-squares contained inside of a subdivided rectangle, also tends to $\infty$ as $n \ra \infty$. This means we may assume there are many $v_n$-squares in the bounded and unbounded complementary components of the core curve of $\Gamma_n$, and the reasoning of the proof still applies.

In fact, since the number of $v_n$-slabs and $v_n$-squares tends to $\infty$, for large enough $n$, the reasoning of Lemma \ref{crossingslabs} allows us to conclude that there are at least three adjacent $v_n$-slabs or $v_n$-squares so that $\sigma$ enters ever single rook of all three $v_n$-slabs or $v_n$-squares. We will use this observation in the basic estimates below.
\begin{lemma}
\label{blockcross}
Suppose that $\sigma$ crosses $\Gamma$ and $\cP(\sigma)$ only contains one point. Then there exists a $v_n$-slab $R$ in some $\Gamma_n$ so that
\begin{enumerate}
	\item $\sigma$ passes through every rook in the $v_n$-slab $R$
	\item $\sigma$ passes through every rook in the adjacent $v_n$-slabs above and below $R$
\end{enumerate}
Moreover, we have
$$l(\sigma \cap R) > \delta_n d_n.$$
\end{lemma}
\begin{proof}
The proof is a picture; see Figure \ref{RookSlab}. The distance $\sigma$ must travel to go between two rooks in consecutive rows in the rook placement is at least $\delta_n d_n/v_n$ distance apart by Lemma \ref{rook}. $\sigma$ must also cross from the top of $R$ to the bottom of $R$, contributing a length of at least $v_n \cdot \frac{2w_n}{t_n}$. So we calculate
$$l(\sigma \cap R) \geq v_n \cdot \frac{2w_n}{t_n} + v_n \cdot \frac{\delta_n d_n}{v_n} > \delta_n d_n.$$
This proves the claim.
\end{proof}

\begin{figure}[!h]
	\centerline{ \includegraphics[height=3in,width=4in]{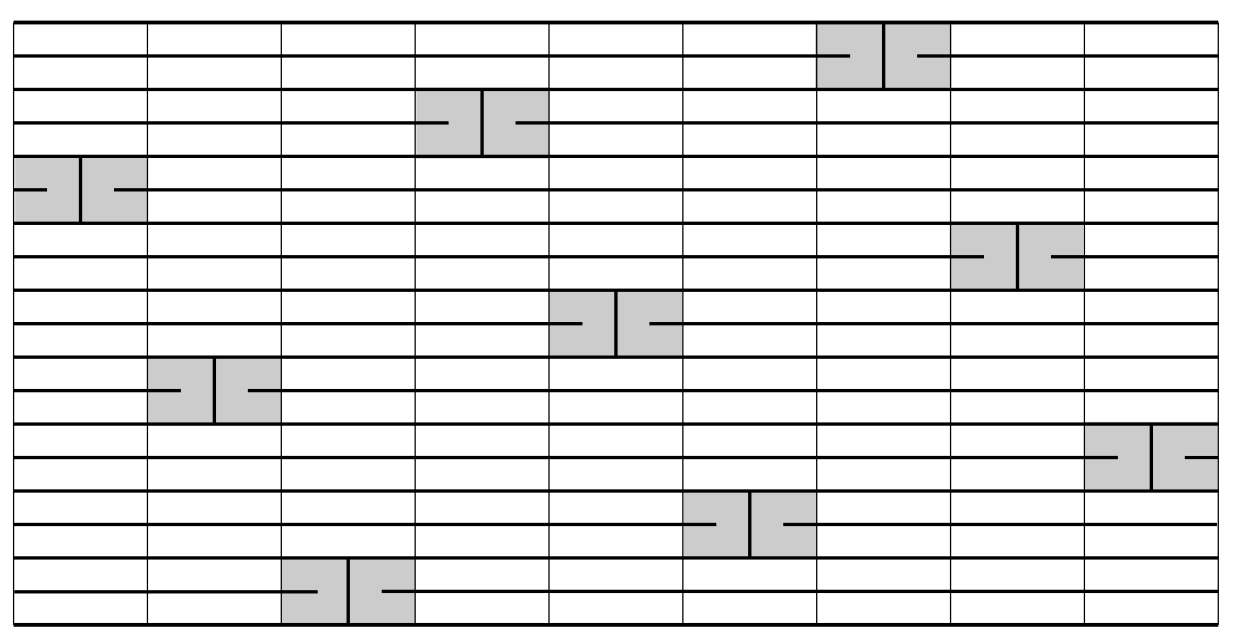} }
	\caption{Illustrating Lemma \ref{blockcross}. We have a $v_n$-slab that is a subset of $\Gamma_{n}$, along with the thin boundary rectangles that form part of the boundary of $\Gamma_{n+1}$. If $\sigma$ crosses $\Gamma_n$ without crossing $\Gamma$, it must enter all the rook squares (shaded in gray) of every $v_n$-slab it passes through.}
	\label{RookSlab}
\end{figure}

\begin{corollary}
	\label{Square1}
Suppose that $\sigma$ crosses $\Gamma$ and $\cP(\sigma)$ only contains one point. Then there exists a $v_n$-square $R$ in some $\Gamma_n$ so that 
\begin{enumerate}
	\item $\sigma$ passes through every rook in the $v_n$-square $R$.
	\item $\sigma$ passes through and every rook of the $v_n$-square above and below $R$.
\end{enumerate}
Moreover, we have
$$l(\sigma \cap R) > 1.$$
\end{corollary}
\begin{proof}
To cross $R$ in such a way, one must pass through all the rooks in $k_n$ many $v_n$-slabs without piercing $\Gamma$. Therefore by repeated applications of Lemma \ref{blockcross}, and then using (\ref{biggerthanone}), we have
$$l(\sigma \cap R) \geq k_n \delta_n v_n = \frac{t_n}{v_n} \cdot \frac{(\delta_n d_n)^2}{2 w_n} > 1.$$ 
\end{proof}

\begin{corollary}
\label{Nonpierceable}
$\Gamma$ is non-pierceable by rectifiable arcs. 
\end{corollary}
\begin{proof}
The follows from immediately from Lemma \ref{Cantor}, since Lemma \ref{Square1} shows that $\Gamma$ is non-piercable by simple rectifiable arcs with length bounded by $1$.
\end{proof}

\section{Any crossing rectifiable curve intersects on positive length}
We would like to focus on a convenient subarc of $\sigma$, which exists by the following lemma.
\begin{lemma}
Suppose $\sigma$ is simple and rectifiable with length $\leq 1$ and crosses $\Gamma$. Then there exists an integer $n \geq 0$ and a straight piece $P$ of $\Gamma_n$ so that $\sigma$ enters the top of $P$ and exits out the bottom of $P$. 
\end{lemma}
\begin{proof}
By Lemma \ref{MustCross}, $\sigma$ must cross $\Gamma_n$ for some $n$. If $\sigma$ did not enter the top and exit out of the bottom of the same straight piece $P$ for $\Gamma_{n+1}$, then similar reasoning to Lemma \ref{crossingslabs} and Lemma \ref{blockcross} would imply that $\sigma$ must enter too many rook squares and must have length greater than $1$.
\end{proof}

We will show that the subarc that begins at the top of a straight piece and exits at the bottom of a straight piece intersects with positive length. We first consider the special case that this subarc is a line segment.

\begin{lemma}
	\label{segment}
If $\sigma$ is a axis aligned line segment that crosses $\Gamma$, then 
$$H^1(\sigma \cap \Gamma) > 0.$$
\end{lemma}
Recall that we always may orient our view so that that this axis aligned line segment is vertical.
\begin{proof}
Let $P$ denote the straight piece that $\sigma$ crosses contained in some $\Gamma_n$. We define 
$$l_m := l(\sigma \cap \Gamma_m).$$
Then $l_m$ is a decreasing sequence and
$$\lim_{m \ra \infty} l_m = H^1(\sigma \cap \Gamma).$$
We will estimate the length lost between each stage, $l_{m} - l_{m+1}$. By our assumptions we have $l_n \geq w_n$.

$\sigma$ passes through $t_n$ many rows in $P$ determined by the plumbing foliation for $\Gamma_n$. The length $l_{n+1}$ can decrease for two reasons. First, if $\sigma$ passes through a rook square, then it may pass through the thickened vertical segment placed in that square, which does not belong to $\Gamma_{n+1}$. The number of rooks a single line segment can pass through is no more than $k_n \lambda_n$ (recall that $k_n$ and $\lambda_n$ are $\ref{kn}$ and $\ref{lambdan}$, respectively). Second, $l_n$ can decrease by $\eta_n$ for every element of the plumbing foliation that $\sigma$ passes through. This means that the amount of length lost can be bounded above by
$$l_{n} - l_{n+1} \leq t_n \cdot \eta_n + k_n\cdot \lambda_n \frac{2w_n}{t_n} \leq \frac{w_n}{100^n \prod_{j=1}^{n-1}t_j} + \frac{w_n}{v_n}.$$
Here we used $(\ref{etan})$, $(\ref{lambdan})$, and $(\ref{kn})$.

Estimating $l_{n+1} - l_{n+2}$ is similar; we just have to count how many straight pieces $\sigma$ crosses in $\Gamma_{n+1}$. This number is no more than $t_n$. Therefore, we can apply the same estimates above and see that
\begin{eqnarray*}
l_{n+1} - l_{n+2} &\leq& t_n \cdot (t_{n+1} \cdot \eta_{n+1}) + t_n \left(\cdot k_{n+1} \cdot \lambda_{n+1}  \cdot \frac{2w_{n+1}}{t_{n+1}} \right) \\
	&\leq& \frac{w_{n+1}}{100^{n+1}\prod_{j=1}^{n-1}t_j} + \frac{t_n w_{n+1}}{v_{n+1}} \\
	&\leq& \frac{w_{n}}{100^{n+1}} + \frac{ w_{n}}{v_{n+1}}
\end{eqnarray*}
In general, we can use this argument to deduce that
\begin{equation}
l_{m} - l_{m+1} \leq w_n \left( \frac{1}{100^{m}} + \frac{1}{v_{m}} \right).
\end{equation}
	
Putting it all together,
\begin{eqnarray*}
l_n - l(\sigma \cap \Gamma) &=& \lim_{m \ra \infty} \sum_{j=n}^m l_{j} - l_{j+1} \\
	&\leq& w_n \sum_{j=n}^{\infty} \left(\frac{1}{100^{j+1}} + \frac{1}{v_j} \right) \\
	&\leq& \frac{w_n}{10}.
\end{eqnarray*}
Therefore, 
$$l(\sigma \cap \Gamma) \geq l_n - w_n/10 \geq \frac{9}{10} w_n.$$
\end{proof}

With a little more care, we can upgrade these observations above to the proof of Theorem 1.1.
\begin{lemma}
	\label{excessrooks}
Suppose $\sigma$ is a simple rectifiable arc with length $\leq 1$ and crosses a straight piece $P$ of $\Gamma_n$. Then $\sigma$ cannot intersect more than $k_n v_n + k_n \lambda_n$ many thickened vertical segments placed in rooks in $\Gamma_n$.
\end{lemma}
\begin{proof}
Since $\sigma $ must pass through at least $k_n \cdot \lambda_n$ many $v_n$-slabs, $\sigma$ can intersect exactly one thickened vertical segment in one rook of each of the $k_n \cdot \lambda_n$-slabs of $\Gamma_n$ without costing any length, just like a vertical segment. Any additional thickened vertical segments that $\sigma$ can intersect come from $\sigma$ going to additional $v_n$-slabs or from $\sigma$ entering multiple thickened vertical segments in other rooks in the same $v_n$-slab.

Suppose that $\sigma$ passes through thickened vertical segments in more than $k_n \cdot \lambda_n$ many $v_n$ slabs. By Lemma \ref{RookPlacement}, each additional slab of this type that $\sigma$ passes through must take at least a length of
$$\frac{\delta_n d_n}{2v_n}.$$
Any additional thickened vertical segments passed through in any of the $v_n$ slabs that $\sigma$ intersects also costs a length of $\frac{\delta_n d_n}{2v_n}$, again by Lemma \ref{RookPlacement}.

This means the amount of additional thickened vertical segments that $\sigma$ can intersect is no more than $2 k_n v_n$, since $(\ref{biggerthanone})$ implies that
$$k_n \delta_n d_n > 1$$
This is not possible if we assume that $\sigma$ has length $\leq 1$.
\end{proof}

If $\sigma$ intersects a thickened line segment in some rook, we will just assume that $\sigma$ passed through the thickened line segment. This causes $l_{n+1}$ to decrease, but fortunately this cannot happen often. Indeed, note that $2k_n v_n$, is a very small percentage of $t_n$, since by $(\ref{kn})$,
$$\frac{2k_nv_n}{t_n} = \frac{\delta_n d_n}{w_n}.$$
This quantity tends to $0$ very rapidly by $(\ref{deltan})$.

This observation allows us to prove Theorem \ref{Main}.

\begin{proof}[Proof of the Main Theorem]
Again, let $P$ be the straight piece of $\Gamma_n$ that $\sigma$ crosses. Again we denote $l_n = l(\sigma \cap \Gamma_n)$. 

$\sigma$ must travel between each adjacent element of the plumbing foliation for $\Gamma_n$. Therefore, to estimate $l_{m} - l_{m+1}$, we can use all of the same estimates from Lemma \ref{segment}; we just have to additionally discard the length of the excess vertical segments that $\sigma$ can visit. But by Lemma \ref{excessrooks}, this number is $2 k_n v_n$, so the amount of excess length discarded is no more than 
$$k_n v_n \cdot \frac{2w_n}{t_{n}} = 2 \delta_n d_n.$$
At the next stage, $\Gamma_{n+1}$, the amount of excess length discarded is no more than
$$t_n k_{n+1} v_{n+1} \cdot \frac{2w_{n+1}}{t_{n+1}} \leq t_n \delta_{n+1} d_{n+1} \leq \frac{w_n}{w_{n+1}}\delta_{n+1}d_{n+1}\leq w_n \delta_{n+1}.$$
We used $(\ref{wntn})$ in the second inequality. Similarly, at the $m$th stage, the amount of excess length discarded compared to a vertical line segment is no more than
$$w_n \delta_{m}.$$
So the amount of excess length discarded compared to a vertical segment is no more than
$$w_n \sum_{j=n}^{\infty} \delta_j < \frac{w_n}{100}.$$
This combined with the estimates in Lemma \ref{segment} show that
$$l(\sigma \cap \Gamma) > w_n\left(\frac{9}{10} - \frac{1}{100}\right) > 0$$
\end{proof}

\bibliographystyle{alpha}
\bibliography{Pdim}

\end{document}